\documentclass[11pt]{amsart}
\usepackage{amsmath,amssymb,amsthm,amscd}
\usepackage{tabularx}
\usepackage{amsfonts,latexsym,euscript}
\usepackage{wasysym}
\usepackage{fancyhdr}

\fancypagestyle{plain}%redefining plain pagestyle
\usepackage{hyperref}
\newtheorem{theorem}{Theorem}

\newtheorem{lemma}[theorem]{Lemma}
\newtheorem{corollary}[theorem]{Corollary}

\newtheorem*{remark}{Remark}

\usepackage{amssymb}
\usepackage[T1]{fontenc}
\usepackage[latin1]{inputenc}
\usepackage[usenames,dvipsnames]{xcolor}
\usepackage{graphicx}
\usepackage{setspace}
\usepackage{txfonts}
\usepackage{longtable}
\usepackage{multirow}
\usepackage{blindtext}
\usepackage{ulem}
\usepackage{bm}
\usepackage[margin=1in]{geometry}% http://ctan.org/pkg/geometry

\AtEndDocument{{\footnotesize%
 \textsc{M.Barowsky, Department of Mathematics, Wellesley College, Wellesley, MA., 02481} \par  
  \textit{E-mail address}, M.~Barowsky: \texttt{mbarowsk@wellesley.edu} \par
  \addvspace{\medskipamount}
  \textsc{W. Damron, Department of Mathematics and Computer Science, Davidson College, Davidson, NC., 28035} \par  
  \textit{E-mail address}, W.~Damron: \texttt{widamron@davidson.edu} \par
   \addvspace{\medskipamount}
  \textsc{A. Mejia, Department of Mathematics, Bard College, Annandale-On-Hudson, NY., 12504} \par
  \textit{E-mail address}, A.~Mejia: \texttt{am8248@bard.edu} \par
 \addvspace{\medskipamount}
  \textsc{F. Saia, Department of Mathematics, Tufts University, Medford, MA., 02155} \par
  \textit{E-mail address}, F.~Saia: \texttt{Frederick.Saia@tufts.edu} \par
 \addvspace{\medskipamount}
  \textsc{N. Schock, Department of Mathematics, California Polytechnic State University San Luis Obispo, San Luis Obispo, CA., 93407} \par
  \textit{E-mail address}, N.~Schock: \texttt{schocknol@gmail.com}\par
  \addvspace{\medskipamount}
  \textsc{K. Thompson, Department of Mathematical Sciences, DePaul University, Chicago, IL., 60614} \par
  \textit{E-mail address}, K. ~Thompson: \texttt{kthompson0721@gmail.com} 
 }}

\title{Classically Integral Quadratic Forms Excepting At Most Two Values}
\author[M.Barowsky]{Madeleine Barowsky}
\author[W.Damron]{William Damron}
\author[A.Mejia]{Andres Mejia}
\author[F.Saia]{Frederick Saia}
\author[N.Schock]{Nolan Schock}
\author[K.Thompson]{Katherine Thompson}
\begin{document}

%%%%%%%%%%%%%%%%%%%%%%%%%%%%%%%%%%%%%%%%%
%% ABSTRACT
%%%%%%%%%%%%%%%%%%%%%%%%%%%%%%%%%%%%%%%%%

\begin{abstract}
Let $S \subseteq \mathbb{N}$ be finite. Is there a positive definite quadratic form that fails to represent only those elements in $S$? For $S = \emptyset$, this was solved (for classically integral forms) by the $15$-Theorem of Conway-Schneeberger in the early 1990s and (for all integral forms) by the $290$-Theorem of Bhargava-Hanke in the mid-2000s. In 1938 Halmos attempted to list all weighted sums of four squares that failed to represent $S=\{m\}$; of his $88$ candidates, he could provide complete justifications for all but one. In the same spirit, we ask, ``for which $S = \{m, n\}$ does there exist a quadratic form excepting only the elements of $S$?'' Extending the techniques of Bhargava and Hanke, we answer this question for quaternary forms. In the process, we prove what Halmos could not; namely, that $x^2+2y^2+7z^2+13w^2$ represents all positive integers except $5$. We develop new strategies to handle forms of higher dimensions, yielding an enumeration of and proofs for the $73$ possible pairs that a classically integral positive definite quadratic form may except.

\end{abstract}
\maketitle

%%%%%%%%%%%%%%%%%%%%%%%%%%%%%%%%%%%%%%%%%%%%%%%%%%%
%%%%%%%%%%%%%%%%%%%%%%%%%%%%%%%%%%%%%%%%%%%%%%%%%%%
%%%%%%%%%%%%%%%%%%%%%%%%%%%%%%%%%%%%%%%%%%%%%%%%%%%
%%%%%%%%%%%%%%%%%%%%%%%%%%%%%%%%%%%%%%%%%%%%%%%%%%%
%%%%%%%%%%%%%%%%%%%%%%%%%%%%%%%%%%%%%%%%%%%%%%%%%%%
%%%%%%%%%%%%%%%%%%%%%%%%%%%%%%%%%%%%%%%%%%%%%%%%%%%
%%%%%%%%%%%%%%%%%%%%%%%%%%%%%%%%%%%%%%%%%%%%%%%%%%%

%%%%%%%%%%%%%%%%%%%%%%%%%%%%%%%%%%%%%%%%%
%% INTRODUCTION
%%%%%%%%%%%%%%%%%%%%%%%%%%%%%%%%%%%%%%%%%

\section{Introduction and Statement of Results}
Determining which integers are represented by a positive definite quadratic form has been of long-standing mathematical interest. One particular phrasing of the question is as follows: given a finite subset $S \subseteq \mathbb N$, do there exist quadratic forms that represent $n \in \mathbb N$ if and only if $n \not\in S$? \\
\\
When $S = \emptyset$, the answer is known; this is the case of classifying \textit{universal} positive definite quadratic forms. In $1916$, Ramanujan \cite{Ramanujan} proved the existence of $55$ universal quaternary diagonal positive definite quadratic forms. While it was later discovered that one of those forms fails to represent $n=15$, the remaining $54$ do give the complete list. Ramanujan's main technique was that of escalation, a crucial tool in later universality results. For classically integral forms, there is the Conway-Schneeberger $15$ Theorem. First announced in 1993, it states that a classically integral positive definite quadratic form is universal if and only if it represents all positive integers up to $15$; moreover, provided was a list of the $204$ quaternary forms with this property. For the more general case, there is the $2005$ result of Bhargava-Hanke: the $290$ Theorem. This states that an integral positive definite quadratic form is universal if and only if it represents all positive integers up to $290$; again, provided was a list of $6436$ such forms in four variables \cite{BH}. \\
\\
A natural next question is to classify those positive definite quadratic forms representing all positive integers with a single exception (i.e.\@\xspace the case when $ S  =\{m\}$). Using Ramanujan's idea of escalation, in \cite{Halmos} Halmos provided a list of $135$ diagonal quaternary candidate forms that could fail to represent at most one integer; using known results about representation by ternary subforms, Halmos was able to narrow that list down to $88$ candidate forms. He provided proofs that $87$ of those did indeed fail to represent exactly one integer, though like Ramanujan there is some overcounting (two of Halmos' $88$ except exactly \textit{two} integers). Halmos suspected the remaining form, $x^2+2y^2+7z^2+13w^2$, only failed to represent $5$, but he could not provide a proof. Pall \cite{Pall} did supply a proof in $1940$, using genus theory and basic modular arithmetic. Continuing the study of forms representing all but finitely many integers, in 2009 Bochnak and Oh \cite{BO} provided necessary and sufficient modular conditions for when a classically integral positive definite quadratic form excepts finitely many values.\\
\\
In this paper, we offer a different set of conditions, analogous to those of Bhargava-Hanke, for the case where a quadratic form excepts precisely one value. To do this, we study via escalation the sets $S = \{m,n\}$ (with $m < n$) for which there exists a classically integral positive definite quadratic form representing precisely $\mathbb{N} - S$. Furthermore, we develop methodology that can be adapted to prove similar theorems for other finite sets of exceptions. Minimal results in this area appear in past literature. Most recent is a proof of Hanke in \cite{Hanke} solving an outstanding conjecture of Kneser which claimed that the form $x^2+3y^2+5z^2+7w^2$ represents precisely those $m \in \mathbb N-\{2,22\}$. \\
\\
We now state our main results.
\begin{theorem}\label{Th1}
There are exactly $73$ sets $S=\{m,n\}$ (with $m<n$) for which there exists a classically integral positive definite quadratic form representing precisely $\mathbb{N} - S$. 
\end{theorem}

Both the $15$ and $290$ Theorems supply more specific lists of \textbf{critical integers} whose representability by a form $Q$ is necessary and sufficient for universality. The criticality of such an integer $m$ is proven by explicitly constructing an almost universal quadratic form excepting only $\{m\}$.

\begin{remark}
With the knowledge that $m \leq 15$, our escalation process independently verifies each of the 9 critical integers $\{1, 2, 3, 5, 6, 7, 10, 14, 15\}$ given by the $15$ Theorem. Furthermore, for all excepted pairs $\{m,n\}$, note that $m$ must be a critical integer.
\end{remark}

The following corollary gives criteria for proving that a given form excepts precisely one critical integer.

\begin{corollary}Given any $m \in \{1,2,3,5,6,7,10,14,15\}$, let $n_{max}$ denote the largest $n$ such that $\{m,n\}$ is a possible excepted pair. If a form $Q$ excepts $m$ but represents all other $k \leq n_{max}$, then $m$ is the only exception for $Q$.
\end{corollary}

We are also able to provide the proof that Halmos could not, using a method different from that of Pall: 
\begin{corollary}
The only positive integer $m$ that is not represented by $x^2+2y^2+7z^2+13w^2$ is $m=5$.
\end{corollary}
We further note the following errors in Halmos' original results:

\begin{corollary}
The forms $x^2+y^2+2z^2+22w^2$ and $x^2+2y^2+4z^2+22w^2$ represent precisely $\mathbb{N} - \{14,78\}$, in contradiction to Halmos' conjecture of $\mathbb{N} - \{14\}$.
\end{corollary}

%%%%%%%%%%%%%%%%%%%%%%%%%%%%%%%%%%%%%%%%%
%% ACKNOWLEDGEMENTS
%%%%%%%%%%%%%%%%%%%%%%%%%%%%%%%%%%%%%%%%%

\textbf{Acknowledgements:} This research was supported by the National Science Foundation (DMS-1461189). We thank Jeremy Rouse for his helpful discussions and suggestions and extend gratitude to both Jonathan Hanke and Jeremy Rouse for providing encouragement and code; specifically, we give thanks to Jonathan Hanke for his Sage and C code implementing many useful quadratic form operations and to Jeremy Rouse for similar code in Magma. Lastly, we are grateful to Wake Forest University for providing access to servers.

%%%%%%%%%%%%%%%%%%%%%%%%%%%%%%%%%%%%%%%%%%%%%%%%%%%
%%%%%%%%%%%%%%%%%%%%%%%%%%%%%%%%%%%%%%%%%%%%%%%%%%%
%%%%%%%%%%%%%%%%%%%%%%%%%%%%%%%%%%%%%%%%%%%%%%%%%%%
%%%%%%%%%%%%%%%%%%%%%%%%%%%%%%%%%%%%%%%%%%%%%%%%%%%
%%%%%%%%%%%%%%%%%%%%%%%%%%%%%%%%%%%%%%%%%%%%%%%%%%%
%%%%%%%%%%%%%%%%%%%%%%%%%%%%%%%%%%%%%%%%%%%%%%%%%%%
%%%%%%%%%%%%%%%%%%%%%%%%%%%%%%%%%%%%%%%%%%%%%%%%%%%

%%%%%%%%%%%%%%%%%%%%%%%%%%%%%%%%%%%%%%%%%
%% BACKGROUND
%%%%%%%%%%%%%%%%%%%%%%%%%%%%%%%%%%%%%%%%%

\section{Background}
Let $n \in \mathbb N$. An \textbf{$\boldsymbol{n}$-ary integral quadratic form} is a homogeneous integral polynomial of degree two of the form $$Q({\vec{x}})=Q(x_1,\ldots,x_n) = \displaystyle\sum_{1 \leq i \leq j \leq n}a_{ij}x_ix_j \in \mathbb Z\left[x_1,\ldots,x_n\right].$$ 
An equivalent way to represent quadratic forms is by symmetric matrices in $\mathcal M_n \left(\frac{1}{2}\mathbb{Z}\right)$; that is, for each $n$-ary integral quadratic form there exists a unique symmetric matrix $A_Q \in \mathcal M_n \left(\frac{1}{2}\mathbb{Z}\right)$ such that 
$$Q({\vec{x}}) = \vec{x}^{\,t} A_Q \vec{x}.$$

When $A_Q \in \mathcal M_n\left(\mathbb{Z}\right)$ (equivalently, when all cross-terms of $Q$ are even) we say that $Q$ is \textbf{classically integral}. We say that an integral $n$-ary quadratic form is \textbf{positive definite} if $Q\left(\vec{x}\right) \geq 0$, with equality if and only if  $\vec{x} = \vec{0}$. Note that $\det(A_Q) >0$ for any positive definite quadratic form $Q$. 
For the remainder of this paper, assume that any ``form'' is a ``classically integral positive definite quadratic form.''\\
\\
Given two $n$-ary forms $Q_1$ and $Q_2$ with respective matrices $A_{Q_1}$ and $A_{Q_2}$, we say that $Q_1$ and $Q_2$ are \textbf{equivalent (over $\boldsymbol{\mathbb{Z}}$)} if and only if there exists a matrix $M \in GL_n(\mathbb Z)$ such that $A_{Q_1} = MA_{Q_2}M^t$.\\
\\
Let $m \in \mathbb N$ and let $Q$ be an $n$-ary form. We say that $m$ is \textbf{represented} by $Q$ if there exists $\vec{x} \in \mathbb Z^n$ so that  
$$Q\left(\vec{x}\right) = m.$$ 
Let 
$$r_Q(m) := \# \left\{ \vec{x} \in \mathbb Z^n \mid Q(\vec{x})=m\right\}$$
denote the \textbf{representation number of $\boldsymbol{m}$ by $\boldsymbol{Q}$.} As mentioned earlier, a form $Q$ is said to be \textbf{universal} if $r_Q(m)>0$ for all $m \in \mathbb N$. In \cite{Halmos}, Halmos used the term {almost universal} to denote a form $Q$ which failed to represent a single $m \in \mathbb N$. In the same spirit, we take \textbf{almost universal} to denote a form representing all $m \in \mathbb N-S$, where $S \subseteq \mathbb{N}$ is finite.

%%%%%%%%%%%%%%%%%%%%%%%%%%%%%%%%%%%%%%%%%
%% BACKGROUND: ESCALATION
%%%%%%%%%%%%%%%%%%%%%%%%%%%%%%%%%%%%%%%%%

\subsection{Escalation}
\label{escalation}
Let $Q$ be an $n$-ary form and $A_{Q}$ be its matrix representation. If $L$ is an $n$-dimensional lattice endowed with inner product $\langle \cdot{,}\cdot \rangle$ such that $Q(\vec{x})=\langle \vec{x},\vec{x} \rangle$, we say that $L$ is the \textbf{associated lattice} for $Q$.

We define the $\textbf{truant}$ of $Q$ to be the least $n \in \mathbb{N}-S$ such that $n$ is not represented by $Q$. It should be noted that this generalizes the definition appearing in earlier literature, most notably in \cite{BH}, where $S = \emptyset$.\\
\\
An \textbf{escalation} of $L$ is any lattice generated by $L$ and a vector whose norm is the truant. An \textbf{escalator lattice} is any lattice that can be obtained by successive escalation of the trivial lattice.\\
\\
Given a form $Q$ of $n-1$ variables with associated matrix $A_Q$ and truant $t$, any escalation by $t$ will result in an $n$-ary form $\widehat{Q}$ with associated matrix $$A_{\widehat{Q}} =\left[ \begin{array}{ccc|c} & & & a_{1n}/2 \\ & A_Q & & \vdots \\ & & & a_{(n-1)n}/2 \\
\hline a_{1n}/2 & \cdots & a_{(n-1)n}/2 & t \end{array}\right].$$

The restrictions that $\det(A_{\widehat{Q}})>0$ and that $a_{in} \in 2\mathbb Z$ for $1 \leq i \leq n-1$ imply that there will be finitely many possible escalations of a quadratic form $Q$ by its truant $t$. Note that the resulting escalation lattice does not depend on the order in which basis vectors are appended.

%%%%%%%%%%%%%%%%%%%%%%%%%%%%%%%%%%%%%%%%%
%% BACKGROUND: MODULAR FORMS
%%%%%%%%%%%%%%%%%%%%%%%%%%%%%%%%%%%%%%%%%

\subsection{Modular Forms}
For details and additional background we refer the reader to \cite{DS}, \cite{Hanke}, and \cite{Thompson}. Any additional specific references will be provided in context. 
\\
\\
Throughout, let $Q$ be a \textit{quaternary} form with associated matrix $A_Q$. We define the \textbf{level} $N=N_Q$ of $Q$ to be the smallest integer such that $N(2(A_{Q}))^{-1}$ is an integral matrix with even diagonal entries.

We define the \textbf{determinant} $D=D_Q$ of $Q$ to be the determinant of $A_Q$. For all primes $p \nmid 2N$, we define the following quadratic character $\chi_Q$: $$\chi_Q(p) = \left( \frac{D}{p} \right).$$ 
Recall that the local normalized form (also called the Jordan decomposition) of $Q$ is $$Q(\vec{x}) \equiv_{\mathbb Z_p} \displaystyle\sum_j p^{v_j}Q_j(x_j),$$ with $\dim(Q_j) \leq 2$ (and in fact, for $p \neq 2$, $\dim(Q_j)=1$). We then define
\begin{center}
$\mathbb S_0 = \{ j \ \vert \ v_j=0 \}$, \hspace{1.0in} $\mathbb S_1 = \{ j \ \vert \ v_j =1\}$, \hspace{1.0in} $\mathbb S_2 = \{ j \ \vert \ v_j \geq 2 \}$,
\end{center}
and we let $s_i = \sum_{j \in \mathbb{S}_i} \dim(Q_j)$.\\
\\
We then define the \textbf{theta series} associated to $Q$ as $$\Theta_Q(z) = 1 + \displaystyle\sum_{m \geq 1} r_Q(m)e^{2 \pi i mz}.$$ 

\begin{theorem}
$\Theta_Q(z)$ is a modular form of weight $2$ over $\Gamma_0(N_Q)$ with associated character $\chi_Q$.
\end{theorem}
\begin{proof}
See \cite[Theorem 2.2, pg. 61]{AZ}.
\end{proof}
Each space of modular forms of fixed weight, level, and character decomposes into the space of cusp forms and the space of Eisenstein series.  Therefore, we write $\Theta_Q(z)=E_Q(z) + C_Q(z)$ where $E_Q(z)$ is Eisenstein and $C_Q(z)$ is a cusp form. We then consider $$r_Q(m) = a_E(m) + a_C(m),$$ where $r_Q(m)$, $a_E(m)$, and $a_C(m)$ respectively  denote the $m^{th}$ Fourier coefficient for $\Theta_Q(z)$, $E_Q(z)$, and $C_Q(z)$.

 The Eisenstein coefficients are  well understood; they are nonnegative and can be computed explicitly. The main tool is due to Siegel \cite{Siegel}. Note that throughout, $p$ denotes a place:
\begin{theorem}
(Siegel) $$a_E(m) = \displaystyle\prod_{p \leq \infty} \beta_p(m)$$ where $$\beta_p(m) = \displaystyle\lim_{U \to \{m\}} \dfrac{\operatorname{Vol}(Q^{-1}(U))}{\operatorname{Vol}(U)},$$ 

where $U$ is an open neighborhood of $m$ in $\mathbb Q_p$.

Specifically for $p \neq \infty$, we have that $$\beta_p(m) = \displaystyle\lim_{v \to \infty} \dfrac{\# \{ \vec{x} \in (\mathbb Z / p^v \mathbb Z)^4 \mid Q(\vec{x}) \equiv m \pmod{p^v} \}}{p^{3v}}.$$
\end{theorem}
To compute $\beta_\infty(m)$, the following result applies:
\begin{theorem}\label{Infinite}
(Siegel)
$$\beta_{\infty}(m) = \dfrac{\pi^2m}{\sqrt{\det(A_Q)}}.$$
\end{theorem}
\begin{proof}
This is simply a special case of Hilfssatz $72$ of \cite{Siegel}.
\end{proof}
Now we recall some terminology from Hanke \cite{Hanke} relevant specifically to later local density computations.\\
\\
Let $R_{{p}^v}(m) := \left\{ \vec{x} \in (\mathbb Z/ {{p}^v} \mathbb Z)^4 \mid Q(\vec{x}) \equiv m \pmod{{p}^v} \right\}$ and set $r_{p^v}(m) := \# R_{{p}^v}(m)$. We say $\vec{x}\in R_{{p}^v}\left(m\right)$ 
\begin{itemize}
\item is of \textbf{Zero type} if $\vec{x} \equiv \vec{0} \pmod{ p}$, in which case we say $\vec{x} \in R_{{p}^v}^{\operatorname{Zero}}\left(m\right)$ with $r_{p^v}^{\operatorname{Zero}}\left(m\right) := \#R_{{p}^v}^{\operatorname{Zero}}\left(m\right)$; \\
\item is of \textbf{Good type} if ${p}^{v_j} x_j \not\equiv 0 \pmod{ p}$ for some $j \in\{1,2,3,4\}$, in which case we say $\vec{x} \in R_{{p}^v}^{\operatorname{Good}}\left(m\right)$ with $r_{p^v}^{\operatorname{Good}}\left(m\right) := \#R_{{p}^v}^{\operatorname{Good}}\left(m\right)$; \\
\item and is of \textbf{Bad type} otherwise, in which case we say $\vec{x} \in R_{{p}^v}^{\operatorname{Bad}}\left(m\right)$ with $r_{p^v}^{\operatorname{Bad}}\left(m\right)$.
\end{itemize}

If $r_{p^v}(m)>0$ for all primes $p$ and for all $v \in \mathbb{N}$, we say that $m$ is \textbf{locally represented}. If $Q(\vec{x})=0$ has only the trivial solution over $\mathbb{Z}_p$, we say that $p$ is an \textbf{anisotropic prime} for $Q$.

In the following theorems, we discuss reduction maps that allow for explicit calculation of local densities. Let the \textbf{multiplicity of a map $f : X \to Y$ at a given $y \in Y$} be $\#\{x \in X \mid f(x) = y\}$. If all $y \in Y$ have the same multiplicity $M$, we say that the map has multiplicity $M$.

\begin{theorem}
We have $$r_{p^{k+ \ell}}^{\operatorname{Good}}(m) = p^{3 \ell} r_{p^k}^{\operatorname{Good}}(m)$$ for $k \geq 1$ for $p$ odd and for $k \geq 3$ for $p=2$. 
\end{theorem}
\begin{proof}
See \cite[Lemma 3.2]{Hanke}.
\end{proof}
\begin{theorem}
The map
\begin{eqnarray*}
\pi_Z : R_{p^k}^{\operatorname{Zero}}(m) & \to & R_{p^{k-2}} \left( \dfrac{m}{p^2} \right) \\
\vec{x} &\mapsto& p^{-1} \vec{x} \pmod{p^{k-2}}
\end{eqnarray*}
is a surjective map with multiplicity $p^4$.
\end{theorem}
\begin{proof}
See \cite[pg. 359]{Hanke}.
\end{proof}

\begin{theorem} \leavevmode
\begin{itemize} 
\item \textit{Bad-Type-I:} This occurs when $\mathbb S_1 \neq \emptyset$ and $\vec{x}_{\mathbb S_1}\not\equiv \vec{0}$. The map 
\begin{eqnarray*}
\pi_{B'} : R_{{p}^k, Q}^{\operatorname{Bad-1}}(m) & \to & R_{{p}^{k-1}, Q'}^{\operatorname{Good}} \left(\dfrac{m}{p} \right)
\end{eqnarray*}
which is defined for each index $j$ by
\begin{center}
$$\begin{array}{ccc} {x_j} \mapsto p^{-1} {x_j} & v_j ' = v_j+1, & j \in \mathbb S_0 \\ {x_j} \mapsto {x_j} & v_j' = v_j-1, & j \not\in \mathbb S_0 \end{array}$$
\end{center}
is surjective with multiplicity $p^{s_1+s_2}$. \\

\item \textit{Bad-Type-II:} This can only occur when $\mathbb S_2 \neq \emptyset$ and involves either $\mathbb S_1 = \emptyset$ or $\vec{x}_{\mathbb S_1} \equiv \vec{0}$. The map
\begin{eqnarray*}
\pi_{B''} : R_{{p}^k, Q}^{\operatorname{Bad-II}}(m) & \to & R_{{p}^{k-2}, Q''}^{\vec{x}_{\mathbb S_2} \not\equiv \vec{0}} \left( \dfrac{m}{{{p}}^2} \right)
\end{eqnarray*}
which is defined for each index $j$ by
\begin{center}
$$\begin{array}{ccc} {x_j} \mapsto  p^{-1} {x_j} & v_j '' = v_j, & j \in \mathbb S_0 \cup \mathbb S_1 \\ {x_j} \mapsto {x_j} & v_j'' = v_j-2, & j \not\in \mathbb S_2 \end{array}$$
\end{center}
is surjective with multiplicity $p^{8-s_0-s_1}$.
\end{itemize}
\end{theorem}
\begin{proof}
Again, see \cite[pg. 360]{Hanke}.
\end{proof}
In practice, instead of exact formulas for $a_E(m)$, we employ effective lower bounds on $a_E(m)$.
Given a form $Q$ with level $N_Q$ and associated character $\chi_Q$, then for all $m$ locally represented by $Q$ we have
\begin{equation*}
a_E(m) \geq C_E m \left(\displaystyle\prod_{\substack{p \nmid N_Q, p \mid m \\ \chi(p) = -1}} \dfrac{p-1}{p+1} \right),
\end{equation*}
where $C_E$ is a highly technical constant achieved from reasonable lower bounds for all $\beta_p(m)$. An exact formula for $C_E$ is given in Theorem $5.7(b)$ of \cite{Hanke}. \\
\\

We similarly compute an upper bound on $\lvert a_C(m) \rvert$, due to Deligne \cite{Deligne}.

For a given $Q$, there exists a constant $C_f$, determined by writing $C_{Q}(z)$ as a linear combination of normalized Hecke eigenforms (and shifts thereof), such that for all $m \in \mathbb{N}$
$$\lvert a_C(m) \rvert \leq C_f  \sqrt{m} \hspace{.05in} \tau(m),$$
where $\tau(m)$ counts the number of positive divisors of $m$. Roughly speaking, $a_E(m)$ is about size $m$ and $a_C(m)$ is about size $\sqrt{m}$.

%%%%%%%%%%%%%%%%%%%%%%%%%%%%%%%%%%%%%%%%%%%%%%%%%%%
%%%%%%%%%%%%%%%%%%%%%%%%%%%%%%%%%%%%%%%%%%%%%%%%%%%
%%%%%%%%%%%%%%%%%%%%%%%%%%%%%%%%%%%%%%%%%%%%%%%%%%%
%%%%%%%%%%%%%%%%%%%%%%%%%%%%%%%%%%%%%%%%%%%%%%%%%%%
%%%%%%%%%%%%%%%%%%%%%%%%%%%%%%%%%%%%%%%%%%%%%%%%%%%
%%%%%%%%%%%%%%%%%%%%%%%%%%%%%%%%%%%%%%%%%%%%%%%%%%%
%%%%%%%%%%%%%%%%%%%%%%%%%%%%%%%%%%%%%%%%%%%%%%%%%%%

%%%%%%%%%%%%%%%%%%%%%%%%%%%%%%%%%%%%%%%%%
%% COMPUTATIONAL AND ANALYTIC METHODS
%%%%%%%%%%%%%%%%%%%%%%%%%%%%%%%%%%%%%%%%%

\section{Computational Methods}
\label{ComputationalMethods}

Following the ideas of \cite{BH} and \cite{Hanke} with the notation of \cite{BH}, we know that any $m \in \mathbb{N}$ is represented by $Q$ if it is locally represented, has bounded divisibility at all anisotropic primes, and if the following inequality holds:
\begin{eqnarray}
\dfrac{\sqrt{m'}}{\tau(m)} \prod_{\substack{p \mid m, p \nmid N \\ \chi(p) = -1}} \dfrac{p-1}{p+1} &>& \dfrac{C_f}{C_E} \;,
\end{eqnarray}
where $m'$ is the largest divisor of $m$ with no anisotropic prime factors.
Define $B(m)$ to be the left side of $(1)$. Note that $B(m)$ is multiplicative. Hence, as $m$ becomes divisible by more primes, $B(m)$ is an increasing function. (This is not true in general; it is possible for primes $p < q$ to have $B(p) > B(q)$, as described below in Lemma \ref{BpLemma}.) Therefore, there are only finitely many $m \in \mathbb N$ where 
\begin{eqnarray}
B(m) &\leq& \dfrac{C_f}{C_E}.
\end{eqnarray}
We call locally represented $m \in \mathbb N$ satisfying $(2)$ \textbf{eligible}.

\subsection{Generating Eligible Numbers}
\label{GeneratingNumbers}
Since $\tfrac{C_f}{C_E}$ can be quite large in comparison to $B(m)$, we require a faster method of computing eligible numbers than sequentially calculating all numbers such that inequality $(2)$ holds. Therefore, we introduce the concept of an {eligible prime}. Define
\[C_B := \prod_{\substack{p \\ B(p) < 1}} B(p).\]
Note that this product will never include primes higher than $7$, since $B(p) > 1$ for any prime $p \geq 11$. 
An \textbf{eligible prime} is then a prime $p$ such that
\begin{eqnarray}
B(p) &\leq& \dfrac{C_f}{C_E C_B}.
\end{eqnarray}
Observe that not all eligible primes are eligible numbers. To generate a list of eligible primes, we enumerate through all primes until $(3)$ no longer holds. We note, however, that we must also check the next prime after that for which $(3)$ fails due to the following result (which was not mentioned in \cite{BH}):
\begin{lemma}
\label{BpLemma}
For any primes $p < q$, if $B(p) > B(q)$ then $q-p \leq 2$.
\end{lemma}

\begin{proof}
The only case needing consideration is if $p$ and $q$ are not anisotropic and $\chi(p) \neq -1$ and $\chi(q) = -1$, that is, if 
\[B(p) = \dfrac{\sqrt{p}}{2} \hspace{.5in} \text{ and } \hspace{.5in} B(q) = \dfrac{\sqrt{q}}{2} \dfrac{q - 1}{q + 1}.\]
Note that if $q \leq p+2$, then
\[\left(\dfrac{B(q)}{B(p)}\right)^2 \leq \dfrac{p^3 + 4p^2 + 5p + 2}{p^3 + 6p^2 + 9p},\]
which is always less than $1$; therefore, $B(p) > B(q)$ when $q-p \leq 2$.

On the other hand, if $q \geq p + 3$ then
\[\left(\dfrac{B(q)}{B(p)}\right)^2 \geq \dfrac{p^3 + 9p^2 + 25p + 21}{p^3 + 8p^2 + 16p},\]
which is always greater than $1$, so $B(p) < B(q)$ if $q-p \geq 3$.
\end{proof}

Once our list of eligible primes is generated, we sort by $B(p)$ in order to implement the following algorithm, which is also outlined in \cite[\S $4.3.1$]{BH}.

Since any squarefree eligible number is the product of a finite number of distinct eligible primes, we take the product of the smallest eligible primes $p_{\ell}$ until $p_1 \cdots p_{\ell+1}$ is not eligible. We then know that any squarefree eligible number will be the product of at most $\ell$ distinct eligible primes.

Set $a := p_1 \cdots p_r$ to be the product of the first $r$ eligible primes, for each $1 \leq r \leq \ell$. While $a$ is eligible, we replace $p_r$ by $p_{r+1}$ and continue replacing this single prime until $a$ is no longer eligible. We then repeat the process, replacing the last two primes $p_{r-1}$ and $p_r$ with $p_r$ and $p_{r+1}$, respectively. 

We continue this either until we either have run out of eligible primes or until we can no longer increment.\\
\\

Once we have a set of squarefree eligible numbers (including $1$), we determine their representability by the form $Q$ using techniques described in more detail in \ref{ComputingRep}. This results in a finite set $S_1$ of squarefree numbers which  fail to be represented by $Q$. However, as $S_1$ is comprised only of squarefree integers, we do not yet have a complete list of possible exceptions; for instance, if $Q$ excepts $2$ and $8$, then $S_1 = \{2\}$. Thus, we construct a new set $S_2$ of exceptions of the form $sp^2$, where $s \in S_1$, $p$ is an eligible prime, and $B(sp^2) \leq \frac{C_f}{C_E}$. We repeat this process, continuing until $S_h = \emptyset$ for some $h$. We then take $\bigcup_{1 \leq i \leq h-1} S_i$ to be the entire set of exceptions. Note that we need not check $B\left(sp^2\right)$ for $p$ larger than our max eligible prime since $B\left(p^j\right) > B\left(p\right)$ for $j \geq 2$, unless $p = 2$ or $p$ is anisotropic (and our largest anisotropic prime is always less than our largest eligible prime); this means that if $sp$ is not eligible, then neither is $sp^2$.

\subsection{Computing Representability}
\label{ComputingRep}
Given a set of eligible numbers, we proceed to check their representability by the form $Q$. The naive approach of simply computing the theta series up to the largest eligible number is infeasible for most forms. We instead check representability using a \textbf{split local cover} of $Q$, that is, a form $Q'$ that locally represents the same numbers as $Q$ and can be written as $Q' = dx^2 \oplus T$ for a minimal $d \in \mathbb{N}$ and ternary subform $T$. Furthermore, $Q'$ has the property that global representation by $Q'$ implies global representation by $Q$. 
Therefore, we check the global representability of each eligible number $a$ by $Q'$ and hence by $Q$ by checking if $a - dx^2$ is globally represented by $T$.
We thus compute the theta series of $T$ up to precision $Y = \lceil 2dc\sqrt{X} \rceil$, where $X$ is the largest eligible number and $c$ is a constant to allow for multiple attempts at checking $a - dx^2$.
Any numbers that fail to be found as represented by $T$ are then handled by computing the theta series of $Q$.\\
\\
Despite the improvements gained by the split local cover, we still encounter memory and speed issues on many forms. To deal with these, we use an \textbf{approximate boolean theta function}, also described in \cite[\S $4.3.2$]{BH}. This stores a single bit for each number up to $Y$ indicating whether or not it is represented by $T$. Additionally, rather than computing the entirety of the theta series, we evaluate $Q$ only at the vectors in the intersection of an appropriately chosen small rectangular prism and the ellipsoid $T\left(\vec{y}\right) \leq Y$. This gives far fewer vectors to check, at the expense of giving potential false negatives for numbers who are represented by vectors outside of the prism.\\
\\
The combined use of the split local cover and approximate boolean theta function significantly improves runtime speed and memory usage. According to \cite{BH}, it requires storing $\sqrt{X}$ bits and has a runtime of $O\left(X^{1/4}\right)$. This is a substantial improvement over the naive method, which stores $X$ bits and takes $O\left(X^2\right)$ time. We saw such improvement first hand; the form
\[3x^2 -2xy + 4y^2 - 4xz -2yz + 6z^2 -2xw + 8yw - 2zw + 7z^2,\]
which fails to represent the pair $\{1,2\}$, took approximately $31$ minutes to run using a split local cover and approximate boolean theta function. By contrast, our systems were unable to handle the memory requirements for computing the form without a boolean theta function. Even using a split local cover and boolean theta function without approximation took $4$ hours and $25$ minutes.

%%%%%%%%%%%%%%%%%%%%%%%%%%%%%%%%%%%%%%%%%%%%%%%%%%%
%%%%%%%%%%%%%%%%%%%%%%%%%%%%%%%%%%%%%%%%%%%%%%%%%%%
%%%%%%%%%%%%%%%%%%%%%%%%%%%%%%%%%%%%%%%%%%%%%%%%%%%
%%%%%%%%%%%%%%%%%%%%%%%%%%%%%%%%%%%%%%%%%%%%%%%%%%%
%%%%%%%%%%%%%%%%%%%%%%%%%%%%%%%%%%%%%%%%%%%%%%%%%%%
%%%%%%%%%%%%%%%%%%%%%%%%%%%%%%%%%%%%%%%%%%%%%%%%%%%
%%%%%%%%%%%%%%%%%%%%%%%%%%%%%%%%%%%%%%%%%%%%%%%%%%%

%%%%%%%%%%%%%%%%%%%%%%%%%%%%%%%%%%%%%%%%%
%% Higher Escalations
%%%%%%%%%%%%%%%%%%%%%%%%%%%%%%%%%%%%%%%%%

\section{Higher Escalations}
\label{HigherEsc}
We now aim to generalize our results by deducing all possible pairs $\left\{m,n\right\}$ for which there exists a classically integral positive definite quadratic form in any number of variables representing exactly the set $\mathbb{N}-\{m,n\}$. We begin by considering the quaternary forms obtained from our previous escalations which except more than $2$ numbers. We classify these into three types:
\begin{itemize}
\item Type $A$: Forms that except only a finite set of numbers;
\item Type $B$: Forms that except infinitely many numbers but have no local obstructions; and
\item Type $C$: Forms that have local obstructions.
\end{itemize}

For a given form of Type $A$, let $L := \{m,n_1,\ldots,n_k\}$ be its ordered set of exceptions. To determine if there is a higher escalation excepting precisely the pair $\{m,n_i\}$ for $1 \leq i \leq k$, we escalate by the truant and check that $\{m,n_i\}$ still fails to be represented. We repeat this process until either $n_i$ is represented, making this pair an impossibility for escalations of this form, or until  $L-\{m,n_i\}$ is represented.\\
\\

Forms of Type $B$ are due to unbounded divisibility by an anisotropic prime $p$ at integers not represented by the form. (All forms we encountered have at most one anisotropic prime.) Consequently, we conjecture that all but finitely many exceptions are contained in 
\[
 \mathcal{F} := \bigcup_{k \in S} F_k,
\]
where $S$ is a finite subset of $\mathbb{N}$ and $F_k=\left\{kp^j \mid  j \in \mathbb{N}\right\}$. We then compute the representability of all eligible numbers not in $\mathcal{F}$ by using the methods of Section \ref{ComputingRep}, resulting in a finite set of exceptions not in $\mathcal{F}$. For each $k \in S$, we observe that escalations by $k$ and $kp$ will suffice for representation of all eligible numbers in $F_k$. However, in practice, a single escalation usually results in a finite set of exceptions, to which we can then apply the methods of Type $A$.

The methods of Type $B$ forms do not directly extend to Type $C$ forms, as it is not immediately clear whether local obstructions only cause finitely many squarefree exceptions. 

Instead, we seek to generalize the notion of the ``$10$-$14$ switch,'' employed by Bhargava and Hanke in \cite[\S $5.2$]{BH}. This technique exploits the commutativity of the escalation process by escalating a ternary form first by the truant(s) of the quaternaries with local obstructions that it generates, and then by the truant of said ternary. In their case, this removes all local obstructions. However, a single switch is not sufficient for our case; therefore, we generalize to multiple potential switches. First, we escalate the quaternary forms as usual, then for each resulting quinary form we search for a quaternary subform with no local obstructions. We find such subforms for all quinary Type $C$ escalations, to which we can then apply the methods of Type $A$ or Type $B$.

%%%%%%%%%%%%%%%%%%%%%%%%%%%%%%%%%%%%%%%%%%%%%%%%%%%
%%%%%%%%%%%%%%%%%%%%%%%%%%%%%%%%%%%%%%%%%%%%%%%%%%%
%%%%%%%%%%%%%%%%%%%%%%%%%%%%%%%%%%%%%%%%%%%%%%%%%%%
%%%%%%%%%%%%%%%%%%%%%%%%%%%%%%%%%%%%%%%%%%%%%%%%%%%
%%%%%%%%%%%%%%%%%%%%%%%%%%%%%%%%%%%%%%%%%%%%%%%%%%%
%%%%%%%%%%%%%%%%%%%%%%%%%%%%%%%%%%%%%%%%%%%%%%%%%%%
%%%%%%%%%%%%%%%%%%%%%%%%%%%%%%%%%%%%%%%%%%%%%%%%%%%

%%%%%%%%%%%%%%%%%%%%%%%%%%%%%%%%%%%%%%%%%
%% PROOFS OF MAIN RESULTS
%%%%%%%%%%%%%%%%%%%%%%%%%%%%%%%%%%%%%%%%%

\section{Proofs of Main Results}
We now state with more detail our main results.
\setcounter{theorem}{0}
\begin{theorem}
There are exactly $73$ sets $S=\{m,n\}$ (with $m<n$) for which there exists a classically integral positive definite quadratic form that represents precisely those natural numbers outside of $S$. 

\renewcommand{\arraystretch}{1.3}
\begin{table}[ht]
\centering
\caption{Possible Excepted Pairs}
\begin{tabular}{|c|c|c|}
\hline
$m$  & $n$ & \normalfont{Minimum Required Variables} \\ 
\hline 
\multirow{2}{*}{$1$} & $2, 3, 4, 5, 6, 7, 9, 10, 11, 13, 14, 15, 17, 19, 21, 23, 25, 26, 30, 41$ & $4$ \\
\cline{2-3}
 & $55$ & $5$ \\
\hline
\multirow{2}{*}{$2$}  & $3,5,6,8,10,11,15,18,22,30,38 $ & $4$  \\
\cline{2-3}
 & $14, 50$ & $5$ \\
\hline
$3$  & $6,7,11,12,19,21,27,30,35,39  $ & $4$  \\
\hline
\multirow{2}{*}{$5$}  & $7,10,13,14,20,21,29,30,35  $ & $4$   \\
\cline{2-3}
 & $37, 42, 125$ & $5$ \\
\hline 
\multirow{2}{*}{$6$}  & $15$ & $4$      \\
\cline{2-3}
 & $54$ & $5$ \\
\hline
$7$  & $10,15,23,28,31,39,55$ & $4$  \\
\hline
\multirow{2}{*}{$10$} & $15,26,40,58$ & $4$ \\
\cline{2-3}
 & $250$ & $5$ \\
\hline
$14$ & $30,56,78$ & $4$  \\
\hline
\end{tabular}
\end{table}

\end{theorem}
\setcounter{theorem}{11}

To generate our candidate forms excepting all pairs $\{m, n\}$, we borrow from the theory of escalator lattices. By the 15 Theorem \cite{Bhargava}, we know that $m \leq 15$, so it remains to find the possible values of $n$.

For quaternary forms, we find the maximum value of $n$ by fixing an $m$ and pursuing our standard method of escalating by a vector with a norm equal to the truant. We implement the escalation process outlined in Section \ref{escalation} using the free and open-source computer algebra system SageMath \cite{Sage} and the \texttt{QuadraticForm} class in particular. Once we reach four-dimensional forms, we use the maximal truant of the forms on our list as the upper bound for $n$. 

When $m = 6$, we find that none of the four-variable candidates generated with our usual method except any additional values. Thus we approach this case by explicitly fixing each $n \leq 15$ (because $15$ is the smallest truant of any ternary escalator excepting $6$) to determine whether any forms excepting $\{6, n\}$ exist.  

For all other $m$, we begin the escalation process anew, this time with a bounded range for $n$ given $m$. Along the way, we remove any forms that represent $m$ or $n$ and any which are not positive definite. At each dimension we also iterate through our list to remove those forms equivalent to another in the list.

These techniques generate exhaustive lists of four-variable candidate forms failing to represent exactly two $x < 10000$, as well as candidate forms excepting precisely one $x < 10000$. To prove that these forms represent all $x \geq 10000$, we implement the methods described in Section \ref{ComputationalMethods} using the Magma computer algebra system \cite{Magma}. This proves that there are $65$ possible excepted pairs by quaternary forms.

To determine a full list of all possible pairs $\{m,n\}$, we next consider higher-dimensional forms. We begin by sorting the quaternary forms (resulting from our prior escalations) that except more than two values into the categories listed in Section \ref{HigherEsc}, and applying the corresponding techniques. Although some quinary forms do except more than two values, escalating to six or more variables yields no new pairs. Hence, we complete our list with $8$ additional pairs $\{m,n\}$ for which there is an almost universal quinary quadratic form.

%%%%%%%%%%%%%%%%%%%%%%%%%%%%%%%%%%%%%%%%%%%%%%%%%%%
%%%%%%%%%%%%%%%%%%%%%%%%%%%%%%%%%%%%%%%%%%%%%%%%%%%
%%%%%%%%%%%%%%%%%%%%%%%%%%%%%%%%%%%%%%%%%%%%%%%%%%%
%%%%%%%%%%%%%%%%%%%%%%%%%%%%%%%%%%%%%%%%%%%%%%%%%%%
%%%%%%%%%%%%%%%%%%%%%%%%%%%%%%%%%%%%%%%%%%%%%%%%%%%
%%%%%%%%%%%%%%%%%%%%%%%%%%%%%%%%%%%%%%%%%%%%%%%%%%%
%%%%%%%%%%%%%%%%%%%%%%%%%%%%%%%%%%%%%%%%%%%%%%%%%%%

%%%%%%%%%%%%%%%%%%%%%%%%%%%%%%%%%%%%%%%%%
%% AN EXAMPLE
%%%%%%%%%%%%%%%%%%%%%%%%%%%%%%%%%%%%%%%%%

\section{An Example}

\subsection{Escalation for \texorpdfstring{$m=5$}{Lg}}
We escalate to construct candidates for quadratic forms that except $\{5,n\}$. This process will also provide a list of all forms which could only fail to represent $5$, and therefore will include Halmos' form. The escalation of the trivial lattice results in the lattice defined by  $\begin{bmatrix}1\end{bmatrix}$, which is simply the quadratic form $x^2$. Since the truant here is $2$, any escalation must be of the form 
 \[ A_{(a)}= \begin{bmatrix} 
 1 & a\\
 a & 2
 \end{bmatrix}
 \]
 where $a \in \mathbb Z$ and $\det\left(A_{(a)}\right)>0$. This forces $a \in \{ 0, \pm 1\}$. Noting
 \begin{align*}
 \begin{bmatrix} 
 -1 & 0\\
 0 & 1
 \end{bmatrix}
 A_{(-1)}
 \begin{bmatrix} 
 -1 & 0\\
 0 & 1
 \end{bmatrix}^{t}
 = A_{(1)},
 \end{align*}
we need only consider the two escalators $A_{(0)}$ and $A_{(1)}$. However, as $A_{(1)}$ represents $5$, we in fact proceed only with the escalator $A_{(0)}$.\\
\\ The truant of $A_{(0)}$ is $7$, so the three dimensional escalators are of the form
 
 \[ A_{(0,b,c)}=\begin{bmatrix}
 1 & 0 & b\\
 0 & 2 & c\\
 b & c & 7
 \end{bmatrix}
 \]
 for $b,c \in \mathbb Z$ and $\det\left(A_{(0,b,c)}\right)>0$. This yields $31$ escalator matrices. Up to equivalence, however, there are only the six below: 
 \begin{center}
 $\begin{array}{ccc}
 \begin{bmatrix}
 1 & 0 & -1\\
 0 & 2 & -3\\
 -1 & -3 & 7
 \end{bmatrix}, & 
 \begin{bmatrix}
 1 & 0 & 0\\
 0 & 2 & -3\\
 0 & -3 & 7
 \end{bmatrix}, & 
 \begin{bmatrix}
 1 & 0 & -1\\
 0 & 2 & -1\\
 -1 & -1 & 7
 \end{bmatrix},\\
 \begin{bmatrix}
 1 & 0 & 0\\
 0 & 2 & -1\\
 0 & -1 & 7
 \end{bmatrix},
 &\begin{bmatrix}
 1 & 0 & -1\\
 0 & 2 & 0\\
 -1 & 0 & 7
 \end{bmatrix}, 
&\begin{bmatrix}
 1 & 0 & 0\\
 0 & 2 & 0\\
 0 & 0 & 7
 \end{bmatrix}
 \end{array}$
 \end{center}
 which have the respective truants

\begin{center}
$ \begin{array}{ccc}
13, & 20, & 13, \\
10, & 13, & 14.
\end{array}$
 \end{center}

Hence, we use these truants to escalate once again, obtaining $166$ quaternary forms up to equivalence. 

%%%%%%%%%%%%%%%%%%%%%%%%%%%%%%%%%%%%%%%%%%%%%%%%%%%
%%%%%%%%%%%%%%%%%%%%%%%%%%%%%%%%%%%%%%%%%%%%%%%%%%%
%%%%%%%%%%%%%%%%%%%%%%%%%%%%%%%%%%%%%%%%%%%%%%%%%%%
%%%%%%%%%%%%%%%%%%%%%%%%%%%%%%%%%%%%%%%%%%%%%%%%%%%
%%%%%%%%%%%%%%%%%%%%%%%%%%%%%%%%%%%%%%%%%%%%%%%%%%%
%%%%%%%%%%%%%%%%%%%%%%%%%%%%%%%%%%%%%%%%%%%%%%%%%%%
%%%%%%%%%%%%%%%%%%%%%%%%%%%%%%%%%%%%%%%%%%%%%%%%%%%
%%%%%%%%%%%%%%%%%%%%%%%%%%%%%%%%%%%%%%%%%
%% HALMOS' FORM
%%%%%%%%%%%%%%%%%%%%%%%%%%%%%%%%%%%%%%%%%

\subsection{Halmos' Form}

We use techniques different from those of Pall \cite{Pall} to prove the previously mentioned conjecture of Halmos \cite{Halmos}.
\setcounter{theorem}{2}
\begin{corollary} \label{HalmosThm}
The diagonal quadratic form $Q(\vec{x}) = x^2 + 2y^2 + 7z^2 + 13w^2$ represents all positive integers except for $5$. 
\end{corollary}
\setcounter{theorem}{11}

%%%%%%%%%%%%%%%%%%%%%%%%%%%%%%%%%%%%%%%%%
%% EISENSTEIN COEFFICIENTS
%%%%%%%%%%%%%%%%%%%%%%%%%%%%%%%%%%%%%%%%%

We first note that our form has level $N_{Q} = 728 = 2^3 \cdot 7 \cdot 13$ and  character $\chi_{Q}(p) = \left(\frac{182}{p}\right)$. Finding all $m \in \mathbb{N}$ to be locally represented, we begin with a series of lemmas to determine $a_E(m)$ explicitly.

\begin{lemma}
\label{LfuncEx}
$L_{\mathbb Q}(2,\chi_{Q}) = \dfrac{213\sqrt{182}\pi^2}{33124}$.
\end{lemma}
\begin{proof}
This follows from \cite[pg. 104]{Iwasawa}. 
\end{proof}

%%%%%%%%%%%%%%%%%%%%%%%%%%%%%%%%%%%%%%%%%
%% Primes not Dividing m or N
%%%%%%%%%%%%%%%%%%%%%%%%%%%%%%%%%%%%%%%%%

\begin{lemma}
For primes $q \neq 2,7,13$ with $q \nmid m$
\begin{equation*}
\beta_q(m) = \left( 1 - \dfrac{\chi_Q(q)}{q^2}\right).
\end{equation*}
\end{lemma}
\begin{proof}
This follows from \cite[Lemma 3.3.2]{Hanke}.
\end{proof}

%%%%%%%%%%%%%%%%%%%%%%%%%%%%%%%%%%%%%%%%%
%% Simplification of Siegel Product
%%%%%%%%%%%%%%%%%%%%%%%%%%%%%%%%%%%%%%%%%

These two lemmas, combined with Siegel's product formula \cite{Siegel}, give
\begin{align*}
a_E(m) &= \beta_{\infty}(m) \beta_2(m) \beta_7{m} \beta_{13}(m) \left( \prod_{2,7,13\neq p \mid m} \beta_p(m)\right) \left( \prod_{2,7,13 \neq q \nmid m} \beta_q(m) \right) \\
&= \dfrac{182m}{213}\beta_2(m)\beta_7(m) \beta_{13}(m)  \left(\prod_{2,7,13\neq p \mid m} \dfrac{\beta_p(m)p^2}{p^2-\chi_Q(p)}\right) \; .
\end{align*}

%%%%%%%%%%%%%%%%%%%%%%%%%%%%%%%%%%%%%%%%%
%% Beta_p Where p Divides N = 728
%%%%%%%%%%%%%%%%%%%%%%%%%%%%%%%%%%%%%%%%%

\begin{lemma} %For $\beta_p$ such that $p \dv N_Q = 728$, we have:

\begin{eqnarray*} %\beta_2
\beta_2(m) &=& \begin{cases} \dfrac{3}{4} \displaystyle\sum_{i=0}^{k-1} \dfrac{1}{2^{2i}} + \quad \! \dfrac{1}{2^{2k}} \begin{cases} 3/4 &\text{ if } \operatorname{ord}_{2}(m) = 2k \;,\; m/2^{2k} \equiv 1, 3 \qquad \;  \pmod{8} \\ 
5/4 &\text{ if } \operatorname{ord}_{2}(m) = 2k \;,\; m/2^{2k} \equiv 5, 7 \qquad \;  \pmod{8} \end{cases} \\
\dfrac{3}{4} \displaystyle\sum_{i=0}^{k} \dfrac{1}{2^{2i}} + \dfrac{1}{2^{2k+1}} \begin{cases} 3/4 &\text{ if } \operatorname{ord}_{2}(m) = 2k + 1 \;,\; m/2^{2k+1} \equiv 1, 3 \pmod{8} \\ 
1/4 &\text{ if } \operatorname{ord}_{2}(m) = 2k + 1 \;,\; m/2^{2k+1} \equiv 5, 7 \pmod{8} \end{cases} \end{cases} \\
%beta_7
\beta_7(m) &= & \begin{cases} \dfrac{48}{49} \displaystyle\sum_{i=0}^{k-1} \dfrac{1}{7^{2i}} + \quad \! \dfrac{1}{7^{2k}} \begin{cases} 8/7 &\text{ if } \operatorname{ord}_{7}(m) = 2k \;,\; m/7^{2k} \equiv 1, 2, 4 \qquad \; \pmod{7} \\ 
6/7 &\text{ if } \operatorname{ord}_{7}(m) = 2k \;,\; m/7^{2k} \equiv 3, 5, 6 \qquad \; \pmod{7} \end{cases}\\
\dfrac{48}{49} \displaystyle\sum_{i=0}^{k} \dfrac{1}{7^{2i}} + \dfrac{1}{7^{2k+1}} \begin{cases} 2/7 &\text{ if } \operatorname{ord}_{7}(m) = 2k+1 \;,\; m/7^{2k+1} \equiv 1, 2, 4 \pmod{7} \\ 
0 &\text{ if } \operatorname{ord}_{7}(m) = 2k+1 \;,\ m/7^{2k+1} \equiv 3, 5, 6 \pmod{7} \end{cases} \end{cases} \\
% beta_13
\beta_{13}(m) &= & \begin{cases} \dfrac{168}{169} \displaystyle\sum_{i=0}^{k-1} \dfrac{1}{13^{2i}} + \quad \! \dfrac{1}{13^{2k}} \begin{cases} 14/13 &\text{ if } \operatorname{ord}_{13}(m) = 2k \;,\; m/13^{2k} \equiv 1, 3, 4, 9, 10, 12  \qquad \! \pmod{13} \\ 
12/13 &\text{ if } \operatorname{ord}_{13}(m) = 2k \;,\; m/13^{2k} \equiv 2, 5, 6, 7, 8, 11 \qquad \; \; \! \pmod{13} \end{cases} \\
\dfrac{168}{169} \displaystyle\sum_{i=0}^{k} \dfrac{1}{13^{2i}} + \dfrac{1}{13^{2k+1}} \begin{cases} 2/13 &\text{ if } \operatorname{ord}_{13}(m) = 2k+1 \;,\; m/13^{2k+1} \equiv 1, 3, 4, 9, 10, 12  \pmod{13} \\ 
0 &\text{ if } \operatorname{ord}_{13}(m) = 2k+1 \;,\; m/13^{2k+1} \equiv 2, 5, 6, 7, 8, 11 \;\;\pmod{13} \end{cases} \end{cases}
\end{eqnarray*}

for $k \in \mathbb{N} \cup \{0\}$.
\end{lemma}

\begin{proof}
We provide details for the claim regarding $\beta_2(m)$ with $\operatorname{ord}_2(m)$ odd. The remaining proofs behave similarly. Additional examples of these computations can be found in \cite{Thompson}.

Suppose that $\operatorname{ord}_2(m)=2k+1$ for $k \in \mathbb{N} \cup \{0\}$. Then there are solutions of the Good, Zero, and Bad-I types:
\begin{align*} 
\beta_{2}(m)=\lim_{v \to \infty} \dfrac{r_{2^v}^{\operatorname{Good}}(m)}{2^{3v}} +\lim_{v \to \infty} \dfrac{r_{2^v}^{\operatorname{Zero}}(m)}{2^{3v}} +\lim_{v \to \infty} \dfrac{r_{2^v}^{\operatorname{Bad-I}}(m)}{2^{3v}}.  
\end{align*}

We compute these individually, beginning with the Bad types: \\

Let $Q^{\prime}$ be the quadratic form $Q^{\prime}(\vec{x})=2x^2+y^2+ 14z^2+26w^2$. Then

\begin{align*} \lim_{v \to \infty} \dfrac{r_{2^v}^{\operatorname{Bad-I}}(m)}{2^{3v}}  &= \lim_{v \to \infty}\dfrac{2 r_{2^{v-1},Q^{\prime}}^{\operatorname{Good}}\left(\dfrac{m}{2}\right)}{2^{3v}}\\
&=\dfrac{r_{2^{7},Q^{\prime}}^{\operatorname{Good}}\left(\dfrac{m}{2}\right)}{2^{23}} \: .
\end{align*}

Note that $r_{2^{7},Q^{\prime}}^{\operatorname{Good}}\left(\frac{m}{2}\right)$ is non-zero only if $\operatorname{ord}_2(m)=1$. Now, for the Good types:

\begin{align*}\lim_{v \to \infty} \dfrac{ r_{2^v}^{\operatorname{Good}}(m)}{2^{3v}}=\dfrac{r_{2^7}^{\operatorname{Good}}(m)}{2^{21}}=\dfrac{3}{4}.\end{align*}

And for the Zero types:

\begin{align*} \lim_{v \to \infty} \dfrac{r_{2^v}^{\operatorname{Zero}}(m)}{2^{3v}} 
&=\lim_{v \to \infty} \dfrac{1}{2^{3v}} \left[\sum_{i=1}^{k} 2^{4i}r_{2^{v-2i}}^{\operatorname{Good}}\left(\dfrac{m}{2^{2i}}\right)+\sum_{i=1}^{k} 2^{4i}r_{2^{v-2i}}^{\operatorname{Bad-I}}\left(\dfrac{m}{2^{2i}}\right)\right]\\
&= \sum_{i=1}^{k}\dfrac{r_{2^7}^{\operatorname{Good}}\left({m}/{2^{2i}}\right)}{2^{2i+21}}+\sum_{i=1}^{k} \dfrac{r_{2^7,Q'}^{\operatorname{Good}}\left({m}/{2^{2i+1}}\right)}{2^{2i+23}}\;.
\end{align*}

Simplifying the sum of Good, Bad, and Zero type solutions yields the above claim.
\end{proof}

%%%%%%%%%%%%%%%%%%%%%%%%%%%%%%%%%%%%%%%%%
%% p divides m, doesn't divide level
%%%%%%%%%%%%%%%%%%%%%%%%%%%%%%%%%%%%%%%%%

\begin{lemma} 
\label{LastDensEx}
For $k \in \mathbb N \cup \{0\}$ and primes $p \neq 2, 7, 13$ such that $p \mid m$
\begin{equation*}
\beta_p(m)\cdot \dfrac{p^2}{p^2-\chi_Q(p)} = \begin{cases} \dfrac{1}{p^{2k}} \left( \dfrac{p^{2k+1}-1}{p-1}\right) &\text{ if } \operatorname{ord}_{p}(m) = 2k \;,\; \chi_{Q}(p) = 1 \\[1em]
\dfrac{1}{p^{2k}} \left( \dfrac{p^{2k+1}+1}{p+1}\right) &\text{ if } \operatorname{ord}_{p}(m) = 2k\;,\; \chi_{Q}(p) = -1 \\[1em]
\dfrac{1}{p^{2k+1}} \left( \dfrac{p^{2k+2}-1}{p-1}\right) &\text{ if } \operatorname{ord}_{p}(m) = 2k+1\;,\; \chi_{Q}(p) = 1 \\[1em]
\dfrac{1}{p^{2k+1}} \left( \dfrac{p^{2k+2}-1}{p+1}\right) &\text{ if } \operatorname{ord}_{p}(m) = 2k+1\;,\; \chi_{Q}(p) = -1.
\end{cases}
\end{equation*}
\end{lemma}

\begin{proof}
See \cite[Lemma 3.3.6]{Thompson}.
\end{proof}

Lemmas \ref{LfuncEx} through \ref{LastDensEx}, along with Theorem \ref{Infinite}, provide a means to calculate $a_E(m)$ for any $m \in \mathbb N$. 

%%%%%%%%%%%%%%%%%%%%%%%%%%%%%%%%%%%%%%%%%
%% CUSP COEFFICIENTS
%%%%%%%%%%%%%%%%%%%%%%%%%%%%%%%%%%%%%%%%%

%\subsection*{The Cusp Coefficients}
Given the above formula for $a_E(m)$, as well as values for $r_Q(m)$, we are able to calculate $a_C(m) = r_Q(m) - a_E(m)$ for any $m$. Using this, we determine that the cuspidal subspace of $\mathcal{M}_2\left(\Gamma_0(728),\chi_Q\right)$ has a dimension of $108$. Computing a basis of normalized Hecke eigenforms for this subspace to determine our cuspidal bound as in \cite[\S $4.2.2$]{BH}, we find

\begin{equation*}
C_f = \sum_i \left|\gamma_i\right| \approx 13.4964\;.
\end{equation*}
Additionally, for the Eisenstein bound we find
\begin{equation*}
C_E = \frac{36}{71} \; .
\end{equation*}
%%%%%%%%%%%%%%%%%%%%%%%%%%%%%%%%%%%%%%%%%
%% PROVING ALMOST UNIVERSALITY
%%%%%%%%%%%%%%%%%%%%%%%%%%%%%%%%%%%%%%%%%

We now prove Corollary \ref{HalmosThm} regarding Halmos' Form $Q$. 
\begin{proof}

%%%%%%%%%%%%%%%%%%%%%%%%%%%%%%%%%%%%%%%%%
%% ELIGIBLE PRIMES / NUMBERS
%%%%%%%%%%%%%%%%%%%%%%%%%%%%%%%%%%%%%%%%%

Having calculated both $C_E$ and $C_f$, we now employ the methods detailed in Section \ref{ComputationalMethods} to compute and check eligible numbers. 
Note that 
\begin{equation*}
C_B = B(2)B(3)B(5),
\end{equation*}
since $B(p) > 1$ for all $p > 5$. With this, we compute that there are $5634$ eligible primes and $343203$ squarefree eligible numbers, the largest of which is $18047039010$. Using the approximate boolean theta function of the split local cover  $Q = x^2 \,\oplus \,\left(2y^2 + 7z^2 + 13w^2\right)$, 
we compute the representability of each of these numbers. This approximation shows that all squarefree eligible numbers except $1$ and $5$ are represented. Computing the full theta series of $Q$ we see that, while $1$ is represented, $5$ indeed  is not. Therefore, we take $S_1 = \{5\}$ to be the set of squarefree exceptions. We hence compute that there are $28$ eligible numbers of the form $5p^2$, and the set of exceptions of this form is $S_2 = \emptyset$. Thus, we have that $S = S_1 = \{5\}$ is the entire set of exceptions for this form, confirming Halmos' conjecture. Our implementation of this entire process takes approximately $2$ minutes and $7$ seconds.
\end{proof}

%%%%%%%%%%%%%%%%%%%%%%%%%%%%%%%%%%%%%%%%%
%% Bibliography
%%%%%%%%%%%%%%%%%%%%%%%%%%%%%%%%%%%%%%%%%

\end{document}